\documentclass[a4paper]{amsart}

\RequirePackage[english]{babel}

\usepackage[OT1]{fontenc}    %
\usepackage{type1cm}         

\usepackage{amsthm}
\RequirePackage{amsmath,amsfonts,amssymb,amsthm}

\RequirePackage[dvips]{graphicx}
\usepackage{enumerate}

  \numberwithin{equation}{section}

\newcommand{\Z}{\mathbb{Z}}
\newcommand{\N}{\mathbb{N}}         
  \newcommand{\R}{\mathbb{R}}         
  
  \newcommand{\Rn}{\R^n}              
  \newcommand{\Le}{\mathcal{L}}       
  \newcommand{\Pa}{\mathcal{P}}       
  \newcommand{\Has}{\mathcal{H}^s}            

\newcommand{\iii}{\mathtt{i}}

\newcommand{\closuredOmega}{\overline{\Omega}_d} 
\newcommand{\intbndryOmega}{\partial\Omega_d} 
\newcommand{\closurerhoOmega}{\overline{\Omega}_\rho} 

  \DeclareMathOperator{\dist}{dist}
  \DeclareMathOperator{\diam}{diam}
  \DeclareMathOperator{\length}{\ell}
  \DeclareMathOperator{\Dim}{Dim}

  \newtheorem{thm}{Theorem}[section]
  \newtheorem{lemma}[thm]{Lemma}

  \theoremstyle{remark}
  \newtheorem{rem}[thm]{Remark}
  \newtheorem{rems}[thm]{Remarks}
  \newtheorem{ex}[thm]{Example}
  \newtheorem{q}[thm]{Question}
  \newtheorem{definition}[thm]{Definition}
  \newtheorem{assu}[thm]{Assumption}

\newcounter{minutes}
\divide\time by 60
\newcounter{hours}
\multiply\time by 60 \addtocounter{minutes}{-\time}

\begin{document}

\author{Riku Kl\'en}
\author{Ville Suomala}
\title{Dimension of the boundary in different metrics}

\address{Department of Mathematics and Statistics\\
FI-20014 University of Turku\\
         Finland}

\address{Department of Mathematical sciences\\
         P.O Box 3000 \\
         FI-90014 University of Oulu\\
         Finland}
\email{riku.klen@utu.fi}
\email{ville.suomala@oulu.fi}


\begin{abstract}
We consider metrics on Euclidean domains $\Omega\subset\R^n$ that are induced by
continuous densities $\rho\colon\Omega\rightarrow(0,\infty)$ and study
the Hausdorff and packing dimensions of the boundary of $\Omega$
with respect to these metrics.
\end{abstract}

\maketitle


\tableofcontents

\section{Introduction}

Let $\Omega\subset\Rn$ be a domain. For $x,y\in\Omega$, we
denote by $d(x,y)$ the internal Euclidean distance between $x$ and
$y$ defined as
\[d(x,y)=\inf_{\gamma}\length(\gamma),\]
where the infimum is taken over all rectifiable curves in $\Omega$ with
endpoints $x$ and $y$ and $\length$ refers to the standard Euclidean
length. It is well known and easy to see that $d$ defines a metric on
$\Omega$ called the internal metric. Furthermore, we may extend
this metric to the internal boundary
$\intbndryOmega= \closuredOmega \setminus{\Omega}$, where
$\closuredOmega$ is the standard metric completion of $\Omega$ with
respect to $d$.

Let $\rho\colon\Omega\rightarrow(0,\infty)$ be a continuous
function. We define the $\rho$-length of a rectifiable curve
$\gamma\subset\Omega$ 
as
\[\length_\rho(\gamma)=\int_{\gamma}\rho(z)|dz|\]
where $|dz|$ denotes integration with respect to arclength. The
$\rho$-distance between $x,y\in\Omega$ is then given by
\[d_\rho(x,y)=\inf_{\gamma}\length_\rho(\gamma),\]
where the infimum is again over all curves
joining $x$ to $y$ in $\Omega$.
This defines a metric on $\Omega$ and as with the internal metric, we
may extend it to the $\rho$-boundary of $\Omega$ defined as
$\partial_\rho\Omega= \closurerhoOmega \setminus\Omega$, where
$\closurerhoOmega$ is the standard metric completion of
$\Omega$ with respect to $d_\rho$. Observe that the internal metric $d$ corresponds to $d_\rho$ for the constant function $\rho\equiv 1$.

Thus, given $\rho$ as above (\emph{a density} in what follows), we
have two complete metric spaces $( \closuredOmega ,d)$ and  
$( \closurerhoOmega ,d_\rho)$ which need not be topologically
equivalent. For simplicity, however, we only deal with cases in which
$\partial_\rho\Omega$ may be
naturally identified with a metric subspace of $\intbndryOmega$. 

In this paper, we will consider $\dim_\rho(\partial_\rho\Omega)$ and
$\Dim_\rho(\partial_\rho\Omega)$, the Hausdorff and packing dimensions
of $\partial_\rho\Omega$ with respect to $d_\rho$ 
(For more comprehensive notation and definitions, we refer to Section
\ref{sec:notation} below). 
Classically, this sort of problems arise in connection to
harmonic measures and the boundary behaviour of conformal maps \cite{Mak, Mak89, P, GM, 
KPT}. In that setting, $\rho=|f'|$ for a conformal map $f$ and
$d_\rho$ corresponds to the internal metric on the image domain. The
Hausdorff dimension, $\dim_\rho(\partial_\rho\Omega)$, has been
analysed also for a much larger collection of so called \emph{conformal
densities} on the unit ball $\mathbb{B}^n\subset\mathbb{R}^n$. See
\cite{BKR, BK, Nie}. 
Although we provide some estimates in the setting of conformal densities, our main goal is to study general densities defined on
John domains in $\Rn$, and to provide tools to estimate
the values of the dimensions $\dim_\rho(\partial_\rho\Omega)$ and $\Dim_\rho(\partial_\rho\Omega)$. Because of this, our methods are perhaps more geometric than analytic.

Given $A\subset \closuredOmega$, we denote by $d(x,A)=\inf_{a\in A}d(x,a)$
the internal distance from $x$ to $A$ and, moreover, abbreviate
$d(x)=d(x,\intbndryOmega)$. Of course, $d(x)$ is just the Euclidean distance to the boundary of $\Omega$.

Let us consider the following simple example: Suppose that
$\Omega\subsetneq\R^n$ has smooth boundary, $-1<\beta<0$, and define a 
density $\rho(x)=d(x)^\beta$. Then it is well known, and easy to
see that $\partial_\rho\Omega$ is a ``snowflake''. More precisely,
$d_\rho(x,y)\approx d(x,y)^{1+\beta}$ for all
$x,y\in\partial_\rho\Omega$. Thus, the effect of $\rho$ on the dimensions
of the boundary is described by a power law
\[\Dim_d(\intbndryOmega)/\Dim_\rho(\partial_\rho\Omega)
=\dim_d(\intbndryOmega)/\dim_\rho(\partial_\rho\Omega)=1+\log\rho(x)/\log d(x).\]
Keeping this example in mind, it is now natural to consider (the upper
and lower) limits of the quantity $\log\rho(y)/\log d(y)$ as $y$
approaches the boundary of $\Omega$. Under sufficient assumptions,
this leads to multifractal type formulas for the dimension of
$\partial_\rho\Omega$. For instance, we obtain the
following result.
\begin{thm}\label{thm:intro}
Let $\Omega\subset\R^n$ be a John domain and $\rho > c > 0$ a density. Suppose that
\[
  i(x)=\lim_{y \in \Omega, \, y\rightarrow x} \frac{\log\rho(y)}{\log d(y)}
\]
exists at all points $x\in\intbndryOmega$ and satisfies $i(x) > -1$. Then
\[
  \dim_\rho(\partial_\rho\Omega) = \sup_{\beta>-1}
  (1+\beta)^{-1}\dim_d(\{x\in\intbndryOmega\,:i(x)\le\beta\}) 
\]
An analogous formula holds for the packing dimension.
\end{thm}

This theorem is a simple special case of a more general result,
Theorem \ref{estimates2}, and it can be used to obtain a formula
for the dimensions $\dim_\rho(\partial_\rho\Omega)$ and
$\Dim_\rho(\partial_\rho\Omega)$ in many situations. 
A generic case is the following: $\Omega=\mathbb{B}^n$,
$C\subset\partial\mathbb{B}^n$ is a Cantor set with $0<\dim_d C<\Dim_d
C<n$ and $\rho(x)=d(x,C)^\beta$ for some $\beta>-1$ (Example
\ref{ex:Cantor}).  

In Theorem \ref{thm:intro}, there is
an annoying 
lack of generality since we have to consider inner limits in
the definition of $i(x)$. The situation is different if we know
that the distance $d_\rho(x,y)$ between points
$x,y\in\partial_\rho\Omega$ is realised along curves that are
``non-tangential''. If the density satisfies a suitable Harnack inequality
together with a Gehring-Hayman type estimate, then it is enough to
consider limits along some fixed cones. For
conformal densities, for instance, we may replace the quantity
$i(x)$ by a radial version
$k(x)=\lim_{t\uparrow1}\log\rho(tx)/\log(1-t)$; see Section
\ref{conf} where we actually consider upper and lower limits as
$t\uparrow1$.  

Section \ref{erq} contains several examples and some open
questions. Most notably, in Example \ref{ex:multifractal} we construct
a new nontrivial example of a conformal density with multifractal type
boundary behavior.

As our results indicate, a careful inspection of the power exponents and the
size of certain sub and super level sets of these quantities can be
used to study the dimensions $\dim_\rho(\partial_\rho\Omega)$ and 
$\Dim_\rho(\partial_\rho\Omega)$. Although the main idea in most of our
results is the same, it is perhaps not possible to find a general
statement which would fit into all, or even most, of the interesting
situations. Often, a suitable case study and a combination of
different ideas is needed in order to deduce the relevant information
(for instance, see Examples \ref{ex:Koch}, \ref{hassu}, and
\ref{ex:multifractal}).  We strongly
believe that the ideas we have used can be applied also elsewhere,
beyond the results of this paper.

\section{Notation}

\label{sec:notation}

Let $\Omega\subset\Rn$ be a domain. For technical reasons, we want to
be able to naturally identify $\partial_\rho\Omega$ with a subset of $\intbndryOmega$. To ensure this, we assume throughout this paper that
for all sequences
$(x_i)$, $x_i\in\Omega$, the following two conditions are satisfied:
\begin{gather}\label{ass1}
\text{If }(x_i)\text{ converges in } \closurerhoOmega ,\text{ then it
  converges in } \closuredOmega.\tag{A1}\\
\label{ass2}
\text{If }(x_i) \text{ converges in }\closuredOmega,\text{ it has at
most one accumulation point in }\partial_\rho\Omega.\tag{A2}
\end{gather}
In other words, \eqref{ass1} means that the identity mapping
$(\Omega,d_\rho) \to (\Omega,d)$ has a continuous extension $f \colon
\closurerhoOmega \to \closuredOmega$ and, furthermore, \eqref{ass2}
means that this $f$ 
 is injective.

\begin{definition}\label{def:density}
\emph{A density} is a continuous function $\rho\colon\Omega\rightarrow(0,\infty)$
satisfying \eqref{ass1} and \eqref{ass2}. For simplicity, we also
require that $\partial_\rho \Omega \neq\emptyset$. 
\end{definition}

Whenever we talk about a curve $\gamma$, we
assume that it is rectifiable, is arc-length parametrized, and that
$\gamma(t)\in\Omega$ for all $0<t<\length(\gamma)$ (the endpoints may
or may not belong to $\intbndryOmega$). Note that the internal length of a curve equals the Euclidean length of the curve.
We say that $\Omega\subset\Rn$ is an $\alpha$-John
  domain for $0<\alpha\le 1$, if there is $x_0\in\Omega$ such that all points $x\in\Omega$ may be
joined to $x_0$ by an $\alpha$-cone, i.e. by a curve
$\gamma$ joining $x$ to $x_0$
such that $d(\gamma(t))\geq \alpha\,t$ for all $0\leq
t\leq\length(\gamma)$. If $\alpha$ is not important, we simply talk about
John domains. 
Let $\gamma\subset\Omega$ be a curve. We say
that $\gamma$ is an $\alpha$-cigar if
\begin{equation}\label{sikari}
d(\gamma(t))\geq\alpha
\min\{t,\length(\gamma)-t\}\text{ for all }0\leq t\leq\length(\gamma).
\end{equation}
For technical purposes, we define an
$\alpha$-distance between points $x,y\in\Omega$ as
\[d_\alpha(x,y)=\inf_{\gamma}\length(\gamma)\]
and this time the infimum is taken over all $\alpha$-cigars $\gamma$
joining $x$ and $y$.
It is easy to see that if $\Omega$ is an $\alpha$-John domain, then any two points
$x,y\in\closuredOmega$ 
may be joined by an $\alpha$-cigar. Thus $d_\alpha(x,y)<\infty$ for
all $x,y\in\closuredOmega$. 
Note however that $d_\alpha$ is not
necessarily a metric since it may be infinite and even if it happens to be finite, it may fail to satisfy the triangle
inequality.

Let $X=(X,d_X)$ be a separable metric space. We denote balls
$B_X(x,r)=\{y\in X\,:\,d_X(y,x)<r\}$ and spheres $S_X(x,r)=\{y\in
X\,:\,d_X(x,y)=r\}$. Given $A\subset X$, we define its $s$-dimensional
Hausdorff and packing measures, $\mathcal{H}^{s}_X(A)$ and
$\mathcal{P}^{s}_X(A)$, respectively, by the following procedure:

\begin{align*}
\mathcal{H}^{s,\varepsilon}_X(A)&=\inf \left\{ \sum_{i=1}^\infty \diam_X(A_i)^s \colon A \subset \bigcup_{i\in\N}
A_i\text{ and }\diam_X(A_i)<\varepsilon\text{ for all }i \right\} ,\\
\Has_X(A)&=\lim_{\varepsilon\downarrow
  0}\mathcal{H}^{s,\varepsilon}_X(A),\\
P^{s,\varepsilon}_X(A)&=\sup \left\{ \sum_{i=1}^{\infty}r_{i}^s \colon
\{B_X(x_i,r_i)\}\text{ is a packing of }A\text{ with }r_i\le\varepsilon
\right\} ,\\
P^{s}_X(A)&=\lim_{\varepsilon\downarrow 0}P^{s,\varepsilon}_X(A),\\
\Pa^{s}_X(A)&=\inf \left\{ \sum_{i=1}^{\infty}P^{s}_X(A_i) \colon A\subset\bigcup_{i=0}^{\infty}
A_i \right\},
\end{align*}
where $0<\varepsilon,s<\infty$ and a packing of $A$ is a disjoint collection of
balls with centres in $A$.
We define the Hausdorff and packing dimensions of
$A\subset X$, respectively, as
\begin{align*}
\dim_X(A)&=\sup\{s\geq 0 \colon \Has_X(A)=\infty\}
=\inf\{s\ge0\,:\,\Has_X(A)=0\},\\
\Dim_X(A)&=\sup\{s\geq 0 \colon \Pa^{s}_X(A)=\infty\} =\inf\{s\ge 0\,:\,\Pa^{s}_X(A)=0\},
\end{align*}
with the conventions $\sup\emptyset=0$, $\inf\emptyset=\infty$.

When the domain $\Omega\subset\Rn$ 
has been fixed, we use all the
notation introduced above with the subscript $d$ when referring to the internal metric.
Moreover, given a density $\rho\colon\Omega\rightarrow(0,\infty)$, 
we use the subscript $\rho$ to refer to the corresponding notions in
terms of the
metric $d_\rho$. For example, given $x\in\closuredOmega$,
$y\in \closurerhoOmega$, and $r>0$ we have
$B_d(x,r)=\{z\in \closuredOmega\,:\,d(z,x)<r\}$ and
$S_\rho(y,r)=\{z\in \closurerhoOmega \,:\,d_\rho(z,y)=r\}$.
We also use the notation $B_\alpha(x,r)$ for balls in terms of the
``distance'' $d_\alpha$.
When referring to ``round'' Euclidean balls we use a subindex $e$, so
$B_e(x,r)=\{y\in\Rn\,:\,|y-x|<r\}$ where $|\cdot|$ is the usual
Euclidean distance. We also denote $\mathbb{B}^n=B_e(0,1)\subset \R^n$
and $S^{n-1}=S_e(0,1)\subset\R^n$.
Observe that if
$A\subset \closurerhoOmega$,
both notations $\diam_d(A)$ and $\diam_\rho(A)$ make sense, since by
\eqref{ass1} and \eqref{ass2}, if $x,y\in A$, then $d(x,y),
d_\rho(x,y)<\infty$ are well defined.

To finish this section, we introduce various limits that are used
later to obtain dimension bounds for $\partial_\rho\Omega$. 
For a domain $\Omega \subset \Rn$, a density $\rho$ and
$x\in\intbndryOmega$, we define
\begin{equation}\label{limiti}
i^-(x)=\liminf_{\underset{y\in\Omega}{y\rightarrow x}}
\frac{\log \rho(y)}{\log d(y)},\quad
i^+(x)=\limsup_{\underset{y\in\Omega}{y\rightarrow
  x}}\frac{\log\rho(y)}{\log d(y)},
\end{equation}
where the limits are considered with respect to the internal
metric. Observe that $i^+(x)\ge-1$ for all $x \in \partial_\rho \Omega$, but $i^{-}(x)$ does not
have to be bounded from below.

For a domain $\Omega \subset \Rn$, a density $\rho$ and $\beta>-1$, we define
\begin{eqnarray}
d^+(\beta) &=& \dim_d \{x\in\partial_\rho\Omega\,:\,i^+(x)\leq \beta\},\label{limitd1}\\
D^+(\beta) &=& \Dim_d \{x\in\partial_\rho\Omega\,:\,i^+(x)\leq \beta\},\\
d^-(\beta) &=& \dim_d \{x\in\partial_\rho\Omega\,:\,i^-(x)\leq \beta\},\\
D^-(\beta) &= & \Dim_d \{x\in\partial_\rho\Omega\,:\,i^-(x)\leq \beta\}\label{limitd4}.
\end{eqnarray}

For a density $\rho$ on
$\mathbb{B}^n$ and $x\in S^{n-1}$, we set
\begin{equation}\label{limitk}
k^-(x)=\liminf_{r\uparrow1}
\frac{\log\rho(rx)}{\log(1-r)},\quad
k^+(x)=\limsup_{r\uparrow1}\frac{\log\rho(rx)}{\log(1-r)}.
\end{equation}
Note that for $\Omega = \mathbb{B}^n$ we have $i^-(x) \le k^-(x) \le k^+(x) \le i^+(x)$ for $x\in S^{n-1}$.

Occasionally, we need to make the following technical assumption for the metric
$d_\rho$:
\begin{assu}\label{assu:gamma}
For each $x\in\partial_\rho\Omega$ and each $\varepsilon>0$, there is
$r>0$ such that for all $y\in B_\rho(x,r)$ there is a curve $\gamma$
joining $x$ to $y$ in $\Omega$ such that $h(\gamma)\ge d(x,y)^{1+\varepsilon}$ and
$\ell_\rho(\gamma)\le d_\rho(x,y)^{1-\varepsilon}$.
\end{assu}
Here $h(\gamma)=\sup_{y\in\gamma}d(y)$ is the maximal distance of
$\gamma$ from the boundary (the ``height'' of $\gamma)$. 
This assumption should be understood as a very mild monotonicity condition with respect to $d(x)$. It is used to obtain dimension lower bounds for the part of  $\partial_\rho\Omega$ where $i^+\ge0$. Close to such points, it is hard to obtain lower estimates for the $\rho$-length of curves that stay very close to $\partial\Omega$.
In fact, if the condition \eqref{assu:gamma} fails, it may happen that $\dim_\rho\partial_\rho\Omega=\Dim_\rho\partial_\rho\Omega=0$ even if $\Omega$ is a half-space, $\dim_d\partial_\rho\Omega>0$, and $i^+$ is uniformly bounded. See Example \ref{ex:assugamma}. 

The assumption \ref{assu:gamma} is a natural generalisation of the Gehring-Hayman condition valid for conformal densities, see
\eqref{GHtype2}.

We summarise our main notation in Table \ref{notation}.
\begin{table}
\begin{tabular}{ll}
  $\rho$ & density: A continuous
    $\rho\colon\Omega\rightarrow(0,\infty)$ satisfying \eqref{ass1}
    and \eqref{ass2}\\ 
  $\ell_\rho(\gamma)$ & the $\rho$-length of a rectifiable curve $\gamma$\\
  $d$ & internal metric\\
  $d_\rho$ & $\rho$-metric\\
  $d_\alpha$ & $\alpha$-distance\\
  $d(x)$ & Euclidean distance from $x$ to the boundary\\
  $B_d(x,r),S_d(x,r)$ & ball and sphere with respect to the internal metric $d$\\
  $B_\rho(x,r),S_\rho(x,r)$ & ball and sphere with respect to $d_\rho$\\
  $B_e(x,r),S_e(x,r)$ & ball and sphere with respect to the Euclidean distance\\
  $B_\alpha(x,r),S_\alpha(x,r)$ & ball and sphere with respect to $d_\alpha$\\
  $\dim_d$ & Hausdorff dimension with respect to $d$\\
  $\Dim_d$ & packing dimension with respect to $d$\\
  $\dim_\rho$ & Hausdorff dimension with respect to $d_\rho$\\
  $\Dim_\rho$ & packing dimension with respect to $d_\rho$\\
  $\closuredOmega$ & metric completion of $\Omega$ with respect to $d$\\
  $\intbndryOmega$ & internal boundary $\closuredOmega \setminus \Omega$\\
  $\closurerhoOmega$ & metric completion of $\Omega$ with respect to $d_\rho$\\
  $\partial_\rho\Omega$ & $\rho$-boundary $\closurerhoOmega\setminus\partial_\rho\Omega$.\\
  $i^\pm$, $k^\pm$ & limits used for dimension bounds
\end{tabular}
\caption{Notation.\label{notation}}
\end{table}

\section{Preliminary lemmas}\label{mainlemmas}

We start by recalling the following simple lemma giving estimates
on expansion and compression behaviour of H\"older type
maps.

\begin{lemma}\label{dimlemmaf}
Suppose that $Z$ and $Y$ are separable metric spaces and let $f\colon Z\rightarrow
Y$, $0<\delta<\infty$ and $X\subset Z$.
\begin{enumerate}
\item\label{dimlemma1} If for each $x\in X$ there are
  $0<r_x,C_x<\infty$ so that
$f(B_Z(x,r))\subset B_Y(f(x),C_x r^\delta)$ for all $0<r<r_x$, then
\begin{align}
\delta\dim_Y(f(X))&\leq \dim_Z(X),\label{hhest}\\
\delta\Dim_Y(f(X))&\leq \Dim_Z(X).\label{ppest}
\end{align}
\item\label{dimlemma2} If for each $x\in X$ there are $0<C_x<\infty$
  and a sequence
  $r_{x,i}>0$ such that $\lim_{i\rightarrow\infty}r_{x,i}=0$ and
$f(B_Z(x,r_{x,i}))\subset B_Y(f(x),C_x r_{x,i}^\delta)$ for all $i$,
then
\begin{equation}
\delta \dim_Y(f(X))\leq\Dim_Z(X).\label{hpest}
\end{equation}
\end{enumerate}
\end{lemma}

\begin{proof}
The proof of \eqref{hhest} is standard.
We give some details for \eqref{ppest} and \eqref{hpest}.

To prove
\eqref{ppest}, we first observe that $X=\bigcup_{n\in\N}
X_{n}$ where
\[X_{n}=\{x\in X\,:\,f(B_Z(x,r))\subset B_Y(f(x), n r^\delta)\text{ for
  all }0<r<1/n\}.\]
Let $0<\varepsilon,s<\infty$, $A\subset X_n$ and suppose that
$B_Y(x_i,r_i)$, $i\in\N$ is a
packing of $f(A)$ so that $r_i<\min\{\varepsilon, n^{1-\delta}\}$ for
each $i$. If $y_i\in A\cap
f^{-1}\{x_i\}$ it follows that
$f(B_Z(y_i,n^{-1/\delta}r_{i}^{1/\delta}))\subset B_Y(x_i,r_i)$ (note that there can be more than one $y_i$ with $f(y_i)=x_i$, choosing any of them will do). Thus,
$B_Z(y_i, n^{-1/\delta}r_i^{1/\delta})$ is a packing of $A$. Letting
$\varepsilon\downarrow 0$, this
implies $P^{s}_{Y}(f(A))\leq n^s P^{s\delta}_{Z}(A)$ for all $A\subset
X_n$. As $A\subset X_n$ is arbitrary, we also get $\mathcal{P}^{s}_Y(f(X_n)\leq
n^s\mathcal{P}^{s\delta}_Z(X_n)$, in particular
$\Dim_Y(f(X_n))\leq\Dim_Z(X_n)/\delta$. The claim \eqref{ppest} now
follows
as $X=\cup_{n\in\N} X_n$.

In order to prove \eqref{hpest}, let
\[X_n=\{x\in X\,:\,f(B_Z(x,r_{x,i}))\subset B_Y(f(x),n
r_{x,i}^\delta)\text{ for some sequence }r_{x,i}\downarrow
0\}.\]
Then $X=\cup_{n\in\N} X_n$. Choose $A\subset X_n$ and fix
$s,\varepsilon>0$. Applying the standard $5R$-covering theorem
(see e.g. \cite[Theorem 2.1]{M}) to the
collection
\[\mathcal{B}=\{B_Y(f(x),n r^\delta)\,:\, x\in A, 0<r<\varepsilon,
f(B_Z(x,r))\subset B_Y(f(x),n r^\delta)\}\]
we find a pairwise
disjoint subcollection $\{B_Y(f(x_i),n r_{i}^\delta)\}_i$ of
$\mathcal{B}$ so that
$f(A)\subset\cup_i B_Y(f(x_i), 5n r_{i}^\delta)$. As $\{B_Z(x_i,
r_i)\}_i$ is a packing of $A$, we get
$\mathcal{H}^{s/\delta,5n \varepsilon^{\delta}}_Y(f(A))\leq (5n)^{s/\delta}
P^{s,\varepsilon}_Z(A)$ and letting $\varepsilon\downarrow 0$,
$\mathcal{H}^{s/\delta}_Y(f(A))\leq (5n)^{s/\delta}P^{s}_Z(A)$. As $A\subset
X_n$ is arbitrary, we also get $\mathcal{H}^{s/\delta}_Y(f(X_n))\leq
(5n)^{s/\delta}\mathcal{P}^{s}_Z(X_n)$ and finally
$\dim_Y(f(X))\leq\Dim_Z(X)/\delta$ since $X=\cup_{n\in\N} X_n$.
\end{proof}

Below, we give a variant of Lemma \ref{dimlemmaf} in terms of the
metrics $d$ and $d_\rho$.

\begin{lemma}\label{dimlemmarho}
Suppose that $\Omega\subset\Rn$ is a domain and
$\rho\colon\Omega\rightarrow (0,\infty)$ is a 
density. Let $A\subset
\partial_\rho\Omega$ and $0 \le \delta \le \infty$.
\begin{enumerate}
\item\label{rholemma1} If
   \[\liminf_{r\downarrow
    0}\frac{\log\left(\diam_\rho(B_d(x,r))\right)}{\log r}\geq\delta\]
  for all $x\in A$, then
  $\delta\dim_\rho(A)\leq\dim_d(A)$ and
  $\delta\Dim_\rho(A)\leq\Dim_d(A)$.
\item\label{rholemma2} If
  \[\liminf_{r\downarrow
    0}\frac{\log\left(\diam_d(B_\rho(x,r))\right)}{\log r}\geq\delta\]
  for all $x\in A$, then
  $\dim_\rho(A)\geq\delta \dim_d(A)$ and
  $\Dim_\rho(A)\geq\delta\Dim_d(A)$.
\item\label{rholemma3} If
   \[\limsup_{r\downarrow
    0}\frac{\log\left(\diam_\rho(B_d(x,r))\right)}{\log r}\geq\delta\]
  for all $x\in A$, then
  $\delta\dim_\rho(A)\leq \Dim_d(A)$.
\item\label{rholemma4} If  \[\limsup_{r\downarrow
    0}\frac{\log\left(\diam_d(B_\rho(x,r))\right)}{\log r}\geq\delta\]
  for all $x\in A$, then
  $\Dim_\rho(A)\geq\delta\dim_d(A)$.
\end{enumerate}
\end{lemma}

\begin{proof}
All the claims \eqref{rholemma1}--\eqref{rholemma4} follow easily from Lemma \ref{dimlemmaf} applied to the mapping
$f\colon( \closurerhoOmega ,d)\rightarrow ( \closurerhoOmega ,d_\rho)$, $x\mapsto x$ and its inverse.
To prove
\eqref{rholemma1}, for instance, fix $\lambda<\delta$. Then for all
$x\in A$, there is
$r_x>0$ so that
$B_d(x,r)\subset B_\rho(x,r^\lambda)$ when
$0<r<r_x$. Thus, Lemma \ref{dimlemmaf} \eqref{dimlemma1} implies
$\lambda\dim_\rho(A)\leq \dim_d(A)$ and $\lambda\Dim_\rho(A)\leq\Dim_d(A)$. Letting
$\lambda\uparrow\delta$, yields \eqref{rholemma1}.
\end{proof}

We end the preliminaries with the following lemma.

\begin{lemma}\label{lemma:john}
For all $0<\alpha\le 1$ and $n\in\N$, there exists constants
$N=N(\alpha,n)\in\N$ and $c=c(\alpha)<\infty$ so that for all $\alpha$-John domains
$\Omega\subset\Rn$ the following holds: For all
$x\in\closuredOmega$ and
$r>0$, there are points $x_1,\ldots, x_N\in\closuredOmega$ so
that
$B_d(x,r)\subset\cup_{i=1}^{N}B_{\alpha/2}(x_i,cr)$.
\end{lemma}

\begin{proof}
For all $y\in B_d(x,r)$, let $\gamma_y$ be an $\alpha$-cone that joins
$y$ to $x_0$, where $x_0\in\Omega$ is a fixed John centre of
$\Omega$. Moreover, we let
\[A_y=\{z\in\Omega\,:\, d(z,\gamma_y(t))<\tfrac{\alpha}{3}t\text{ for
  some }0<t<\length(\gamma_y)\}.\]
We may assume that $d(x,x_0)\geq 2r$ since otherwise $B_d(x,r)\subset
B_\alpha(x_0,2r)$.

We first claim that if $y,z\in B_d(x,r)$ such that $B_d(x,2r)\cap A_y\cap
A_z\neq\emptyset$, then
$y$ and $z$ may be joined by an $(\alpha/2)$-cigar
$\gamma$ with $\length(\gamma)\leq c(\alpha)r$. For this, we may
assume that $d(x)<2r$ as otherwise the Euclidean line segment joining
$y$ to $z$ suites as $\gamma$.
Assume that $w\in B_d(x,2r)\cap A_y\cap A_z$ and choose
$t_y,t_z>0$ so that $d(w,\gamma_y(t_y))<\tfrac{\alpha}{3}t_y$ and
$d(w,\gamma_z(t_z))<\tfrac{\alpha}{3}t_z$. Let $\gamma$ denote
the curve which consists of $\gamma_y|_{0<t\leq t_y}$,
$\gamma_z|_{0<t<t_z}$ and the two (Euclidean) line segments joining $w$ to
$\gamma_y(t_y)$ and $\gamma_z(t_z)$. As
$(t_y+\tfrac{\alpha}{3}t_y)\tfrac\alpha2\leq \alpha t_y-\tfrac\alpha3
t_y$ (and similarly for $t_z$), it follows that $\gamma$ is an
$\tfrac\alpha2$-cigar. Now $B_e(w,\tfrac23\alpha
t_y)\subset\Omega$, $B_e(w,\tfrac23\alpha
t_z)\subset\Omega$ by the $\alpha$-cone condition. Combining this with
the fact $d(w)\le |w-x|+d(x)\le 4r$ implies $t_y,t_z\leq\tfrac6\alpha r$ and
consequently 
\[\length(\gamma)\leq \left(1+\frac\alpha3\right)\left(t_y+t_z\right)\leq
\left(\frac4\alpha+1\right)r=c(\alpha)r.\]

Let $x_1,\ldots, x_N\in B_d(x,r)$ be such that $B_d(x,2r)\cap A_{x_j}\cap
A_{x_i}=\emptyset$ whenever $i\neq j$. It suffices to show that $N\leq
N(n,\alpha)$. For each $i$, let $y_i=\gamma_{x_i}(r)$.
Then $B_d(y_i,\alpha
r/3)=B_e(y_i,\alpha r/3)\subset A_{x_i}\cap B_d(x,2r)$ and a volume
comparison yields
$N (r \alpha/3)^n\leq 2^n r^n$ implying the claim for $N(n,\alpha) = (6/\alpha)^n$.
\end{proof}

\begin{rem}
A subset of the boundary of a
John domain has the same Hausdorff dimension both in the internal and
the 
Euclidean metric. Indeed, it follows as in the above proof that for
any $x$ which is an Euclidean boundary point of $\Omega$, the
set $B_e(x,r)\cap\Omega$ may be covered by $N=N(\alpha,r)$
balls of radius $c(\alpha)r$ in the internal metric. A slightly more
detailed argument 
implies a similar statement for the packing
dimension.
\end{rem}

\section{Dimension estimates on general domains}\label{generalestimates}

We first derive some straightforward dimension bounds
arising from the local power law behaviour of the density $\rho$
near $\partial\Omega$. For the definition of $i^{-}(x)$ and $i^{+}(x)$
recall \eqref{limiti}. The relevant assumptions are slightly different
for the upper and lower bounds, and also depend on the sign of
$i^\pm$. Roughly speaking, the positive values of $i^\pm$ correspond to
expansion behaviour (of $d_\rho$ compared to $d$), whereas the
negative values are related to compression of dimensions. If we aim to find
the exact values of $\dim_\rho(\partial_\rho\Omega)$ and
$\Dim_\rho(\partial_\rho\Omega)$, then we are usually more interested
in the set where $i^\pm$ are negative.

\begin{lemma}\label{thm:estimates1}
Suppose that $\Omega\subset\Rn$, $\rho$ is a density
on $\Omega$,
$\beta>-1$,
$A\subset\{x\in\partial_\rho\Omega\,:\,i^{+}(x)\leq\beta\}$ and
$B\subset\{x\in\partial_\rho\Omega\,:\,i^{-}(x)\geq\beta\}$.

If $\beta<0$ or if Assumption \ref{assu:gamma} holds, then
\begin{enumerate}
\item\label{1}
  $(1+\beta)\dim_\rho(A)\geq
  \dim_d(A)$,
\item\label{3}
  $(1+\beta)\Dim_\rho(A)\geq
  \Dim_d(A)$.
\newcounter{enumi_saved}
\setcounter{enumi_saved}{\value{enumi}}
\end{enumerate}

If $\Omega$ is a John domain, or if $\beta>0$, we have
\begin{enumerate}
\setcounter{enumi}{\value{enumi_saved}}
\item\label{2}
  $(1+\beta)\dim_\rho(B)\leq
  \dim_d(B)$,
\item\label{4}
  $(1+\beta)\Dim_\rho(B)
  \leq\Dim_d(B)$.
\end{enumerate}
\end{lemma}

\begin{proof}
Assume first that $\beta<0$ and choose $\beta<s<0$. 
Now, for all $x\in A$, there is $q>0$ so that
$\rho(y)>d(y)^s$ for all $y\in B_d(x,q)$.
Let $r<(q/2)^{1+s}$ 
and choose $y\in
B_\rho(x,r)$ such
that $d(x,y)>\diam_d(B_\rho(x,r))/3$. Also, let $\gamma$ be a curve
joining $x$ to
$y$ such that $\length_\rho(\gamma)<r$.
Then $\gamma \subset B_d(x,q)$ as otherwise there is a curve $\gamma' \subset \gamma \cap \overline{B}_d(x,q)$ connecting $x$ to $\partial B_d(x,q)$, and then
\begin{equation*}
\length_\rho(\gamma) \ge \length_\rho(\gamma')  = \int_{\gamma'}\rho(z)|dz| \ge \int_{\gamma'} d(\gamma(t))^s dt \ge \ell(\gamma')^{1+s} \ge q^{1+s},
\end{equation*}
which is impossible.
Now
$\rho(z)>d(z)^s$ for all $z\in\gamma$ and combining this with the fact
$\length(\gamma)\geq d(x,y)$, we obtain
\begin{equation*}
r >
\length_\rho(\gamma)=\int_{\gamma}\rho(z)|dz|>\left(\frac{d(x,y)}{2}\right)^{1+s}.
\end{equation*} 
This yields
$\diam_d(B_\rho(x,r))<3 d(x,y)<6 r^{1/(1+s)}$.
As this holds for all $0<r<(q/2)^{1/(1+s)}$, we get
\begin{equation}\label{tamakinviela}
\liminf_{r\downarrow
  0}\frac{\log\diam_d(B_\rho(x,r))}{\log r}\geq
\frac1{1+s}
\end{equation}
for all $x\in A$. 

Assume now that  $s>\beta\ge0$ and that Assumption \ref{assu:gamma}
holds. Let $x\in A$, $\varepsilon>0$, $y\in \closurerhoOmega$. If
$d_\rho(x,y)$ is small, then Assumption \ref{assu:gamma} gives a curve
$\gamma$ joining $x$ to $y$ with $h(\gamma)\ge d(x,y)^{1+\varepsilon}$
and $\ell_\rho(\gamma)\le d_\rho(x,y)^{1-\varepsilon}$. Thus, for $r>0$ small
enough, and all $y\in B_\rho(x,r)$, we have
\begin{align*}
d_\rho(x,y)^{1-\varepsilon}\ge\length_\rho(\gamma) =
\int_{\gamma}\rho(z)|dz|\ge h(\gamma)(h(\gamma)/2)^{s}\ge 2^{-s}d(x,y)^{(1+s)(1+\varepsilon)}
\end{align*}
for some curve joining $x$ and $y$.
This shows that under Assumption \ref{assu:gamma},
\eqref{tamakinviela} holds true also if $\beta\ge 0$. 
The claims \eqref{1} and \eqref{3} now follow using Lemma
\ref{dimlemmarho} \eqref{rholemma2} and letting $s\downarrow\beta$.

In order to prove the claims \eqref{2} and \eqref{4}, in view of Lemma \ref{dimlemmarho}
\eqref{rholemma1}, it suffices to show that
\begin{equation}\label{diamest}
\liminf_{r\downarrow
    0}\frac{\log\diam_\rho(B_d(x,r))}{\log r}\geq 1+\beta
\end{equation}
for all $x\in B$. Let $x\in B$ and  $s<\beta$. Then there is $q>0$ so that
$\rho(y)<d(y)^s$ for all $y\in B_d(x,q)$. Let $r<q/(2c+1)$, where
$c=c(\alpha)<\infty$ is the constant of Lemma \ref{lemma:john} and
where $\alpha>0$ is such that $\Omega$ is an $\alpha$-John domain. By Lemma
\ref{lemma:john}, we find $x_1,\ldots, x_N\in  B_d(x,(c+1)r)$,
$N=N(n,\alpha)$, such that
$B_d(x,r)\subset\bigcup_{i=1}^{N}B_{\alpha/2}(x_i,cr)$.

Let $x_i\in\{x_1,\ldots,x_N\}$ and $y\in
B_{\alpha/2}(x_i,cr)$. Then there is
an $(\alpha/2)$-cigar
$\gamma$ joining $x_i$ to $y$ with $\length(\gamma)<c r$. Assume that
$s\le 0$. Since
$r<q/(2c+1)$, we have
$\rho(\gamma(t))<d(\gamma(t))^s
\leq\alpha^s\min\{t,\length(\gamma)-t\}^s$
for all $0<t<\length(\gamma)$ and thus
\[d_\rho(x_i,y)\leq \int_{\gamma}\rho(z)|dz|\leq
2\alpha^s\int_{t=0}^{\length(\gamma)/2}t^s\,
dt=c_1\length(\gamma)^{1+s}<c^{1+s}c_1
r^{1+s}\]
giving $\diam_\rho(B_{\alpha/2}(x_i,r))\leq 2\cdot c^{1+s}
c_1 r^{1+s}=c_2
r^{1+s}$, where $c_2<\infty$ depends only on
$\alpha$, $n$, and $s$. 
As
$B_d(x,r)$
is connected, we arrive at
\begin{equation}\label{poo}
\diam_\rho(B_d(x,r))\leq\sum_{i=1}^{N}\diam_\rho(B_{\alpha/2}(x_i,cr))\leq
N c_2 r^{1+s}.
\end{equation}

If $s\ge0$, we arrive at the same estimate by using the trivial
estimate $\rho(z)\le\ell(\gamma)^s$ for all $z\in\gamma$. 
Since \eqref{poo} holds for all sufficiently small $r>0$ and $s<\beta$
is arbitrary, we get \eqref{diamest}.
\end{proof}

Next we will use the Lemma \ref{thm:estimates1} to obtain
multifractal type formulas for estimating the dimension of
$\partial_\rho\Omega$. 
To recall the definitions of $d^\pm(\beta)$ and $D^\pm(\beta)$, see \eqref{limitd1}-\eqref{limitd4}.

\begin{thm}\label{estimates2}
Let $\Omega\subset\R^n$ be a John domain, and $\rho$ a density on
$\Omega$ so that Assumption \ref{assu:gamma} is satisfied. Then
\begin{align}
\dim_\rho(\partial_\rho\Omega)
\geq\sup\limits_{\beta>-1}\frac{d^+(\beta)}{1+\beta},\label{est21}\\
\Dim_\rho(\partial_\rho\Omega)\geq\sup\limits_{\beta>-1}\frac{D^+(\beta)}{1+\beta},\label{est22}\\
\dim_\rho\left(\partial_\rho\Omega\cap\{x\,:\,i^{-}(x)>-1\}\right)&\leq
\sup\limits_{\beta>-1}\frac{d^-(\beta)}{1+\beta},\label{est23}\\
\Dim_\rho\left(\partial_\rho\Omega\cap\{x\,:\,i^{-}(x)>-1\}\right)&\leq
\sup\limits_{\beta>-1}\frac{D^-(\beta)}{1+\beta}.\label{est24}
\end{align}
\end{thm}

\begin{proof}
Let us prove \eqref{est21} and \eqref{est23}. The other estimates are
obtained similarly
with the help of the corresponding statements of Lemma \ref{thm:estimates1}.
Let 
\[s< \sup\limits_{\beta>-1} \frac{d^+(\beta)}{1+\beta}\] 
and pick $\beta>-1$ such that
$\dim_d \{x\in\partial_\rho\Omega\,:\,i^+(x)\le \beta\}>s(1+\beta)$.
Combining this with  Lemma
\ref{thm:estimates1} \eqref{1} gives
\begin{align*}
\dim_\rho\left(\partial_\rho\Omega\right) \geq\dim_\rho\{x\in
  \partial_\rho\Omega\,:\,i^{+}(x)\le \beta\}
\geq \frac{\dim_d \{x\in
  \partial_\rho\Omega\,:\,i^{+}(x)\le \beta\}}{1+\beta}
>s
\end{align*}
proving \eqref{est21}.%

To prove \eqref{est23}, we observe that given an interval
$[a,b]\subset (-1,\infty)$, Lemma \ref{thm:estimates1} \eqref{2} gives
\begin{align*}
\dim_\rho\{x\in\partial_\rho\Omega\,:\,i^{-}(x)\in[a,b]\}
&\leq\dim_d \{x\in\intbndryOmega\,:\,i^{-}(x)\in[a,b]\}/(1+a)\\
&\leq
\frac{1+b}{1+a}\sup_{\beta>-1}\frac{d^-(\beta)}{1+\beta}. 
\end{align*}
For any
$\varepsilon>0$ we may cover the interval $(-1,\infty)$ with intervals
$[a_i,b_i]_{i\in\N}$ so that $1+b_i<(1+\varepsilon)(1+a_i)$ for all
$i$. Consequently,
\begin{align*}
\dim_\rho(\partial_\rho\Omega\cap\{x\,:\,i^{-}(x)>-1\})&\leq
\sup_{i\in\N}\dim_\rho(\partial_\rho\Omega\cap\{x\,:\,i^{-}(x)\in[a_i,b_i]\})
\\&<
(1+\varepsilon)\sup_{\beta>-1}\frac{d^{-}(\beta)}{1+\beta}.
\end{align*} 
Now \eqref{est23}
follows as $\varepsilon\downarrow 0$.
\end{proof}

\begin{rems}
\textit{a)}
Suppose that $\Omega$ is a John domain, $\rho$ satisfies Assumption
\ref{assu:gamma} and $i^-(x)>-1$ for
all $x\in\partial_\rho\Omega$.
Then Theorem \ref{estimates2} gives a formula for calculating
$\dim_\rho(\partial_\rho\Omega)$ provided that
$\sup_{\beta>-1} d^+(\beta)/(1+\beta)$ and $\sup_{\beta>-1} d^-(\beta)/(1+\beta)$
coincide. In particular, this is the case if $-1<i^-(x)=i^+(x)$  for
all $x\in\partial_\rho\Omega$.
A similar statement is, of course, true for the packing
dimension.
See also the examples below.\\
\textit{b)} In general it is not possible to control
$\dim_\rho\{x\in\partial_\rho\Omega\,:\,i^{-}(x)\leq-1\}$ in terms
of $\dim_d \{x\in\partial_\rho\Omega\,:\,i^{-}(x)\leq-1\}$.
Let $\Omega=\mathbb{B}^n$ and choose a continuous
$f\colon(0,\infty)\to(0,\infty)$ such that $\int_{t=0}^1
f(t)\,dt<\infty$ and $\log f(t)/\log t\to-1$ as $t\downarrow
0$. Then it is possible to construct a Cantor set $C\subset S^{n-1}$
such that $\dim_d (C)=0$ and $\dim_\rho(C)=\infty$ for
$\rho(x)=f(\dist(x,C))$. See also
\cite[Proposition 7.1]{BKR}, where a similar type of example is considered.
\end{rems}

In the following example, all four of the inequalities
\eqref{est21}--\eqref{est24} hold with equalities.

\begin{ex}\label{ex:Cantor}
Let $\Omega=B_e(0,1)\subset\Rn$ and let $C\subset S^{n-1}$ be a Cantor
set with $\dim_d C=s$ and $\Dim_d C=t$. Let $\beta>-1$ and
$\rho(x)=d(x,C)^\beta$. Then $\dim_\rho(C)=s/(1+\beta)$,
$\Dim_\rho(C)=t/(1+\beta)$, and $\dim_\rho(\intbndryOmega\setminus
C)=\Dim_\rho(\intbndryOmega\setminus C) = n-1$. Thus
$\dim_\rho(\partial_\rho\Omega)=\max\{n-1,s/(1+\beta)\}$ and
$\Dim_\rho(\partial_\rho\Omega)=\max\{n-1, t/(1+\beta)\}$.
\end{ex}

Below, we construct an example to show that all inequalities in
Theorem \ref{estimates2} can be strict.

\begin{ex}
There exist domains $\Omega$ and densities $\rho$ such that all four
of the inequalities \eqref{est21}-\eqref{est24} are strict. 

Let $\Omega=\{(x,y)\in\R^2\,:\,y>0\}$ be the upper half-plane and
fix $-1<q<s<p<0$.
Define $A_k=\{(n2^{-2k},2^{-2k})\,:\,n\in\Z\}$,
$B_k=\{(n2^{-2k+1},2^{-2k+1})\,:\,n\in\Z\}$, and
$r_k=2^{-100k^2}$ for all $k\in\N$. Then choose a continuous density
$\rho\colon\Omega\rightarrow(0,\infty)$ so that $\rho(z)=2^{-2kq}$ if
$z\in A_k$, $\rho(z)=2^{-(2k+1)p}$ if $z\in B_k$ and $\rho(z)=d(z)^s$
if $z\in\Omega\setminus \left( \cup_{k\in\N}\cup_{x\in A_k\cup
  B_k}B_d(x,r_k) \right)$. Then $i^+(x)\ge p$ and $i^-(x)\le q$ for all
$x\in\intbndryOmega$. Thus
\begin{align*}\sup_{\beta>-1}d^+(\beta)/(1+\beta)=\sup_{\beta>-1}D^+(\beta)/(1+\beta)\le1/(1+p),\\ 
\sup_{\beta>-1}d^-(\beta)/(1+\beta)=\sup_{\beta>-1}D^-(\beta)/(1+\beta)\ge1/(1+q).
\end{align*}
On the
other hand, it is easy to see that
$\dim_\rho(\partial_\rho\Omega)=\Dim_\rho(\partial_\rho\Omega)=1/(1+s)$.
\end{ex}

Our next example shows that the claims \eqref{1} and \eqref{3} of Lemma \ref{thm:estimates1} do not necessarily hold without the Assumption \ref{assu:gamma}.

\begin{ex}\label{ex:assugamma}
Let $0<\alpha_n<1$ be a sequence satisfying $\sum_{n=1}^\infty\alpha_n<\infty$. We construct a Cantor set $C\subset[0,1]$ with the following procedure: Let $I_\varnothing=[0,1]$, $\ell_0=1$, $I_0=[0,(1-\alpha_1)/2]$, $I_1=[(1+\alpha_1)/2,1]$ and $\ell_1=(1-\alpha_1)/2$. Suppose $n\in\N$, $\iii\in\{0,1\}^n$, and  that $I_\iii$ with $\diam(I_\iii)=\ell_n$ has been defined. We then define inductively $I_{\iii0}$ and $I_{\iii1}$ to be the subintervals of $I_\iii$ with length $\ell_{n+1}=\ell_n(1-\alpha_n)/2$ such that $I_{\iii0}$ has the same left endpoint as $I_\iii$ and $I_{\iii1}$ has the same right endpoint as $I_\iii$. We also denote by $J_\iii$ the interval between $I_{\iii0}$ and $I_{\iii1}$. The $(\alpha_n)$-Cantor set $C=C(\alpha_n)$ is then defined as
\[C=\bigcap_{n\in\N}\bigcup_{\iii\in\{0,1\}^n}I_\iii\,.\]

For each $n\in\N$, we may choose $0<h_n<\ell_n$ such that
\begin{equation}\label{dim0estimate}
2^n\ell_{n}^{1/n}h_{n}^{1/n}\le1\,.
\end{equation} 
We also require that $h_{n+1}\le h_{n}$.

Next we define a density $\rho$ on the upper half-plane $H$. For each $n$, and $\iii\in\{0,1\}^n$, let $T_{\iii}$ and $U_\iii$ be the
isosceles triangles with base $J_\iii$ and heights $h_n$ and $h_n/2$
respectively. For $\iii=\varnothing$, we define $T_\varnothing=\{(x,y)\in H\,:\,x<0\text{ and }y<-2x\}\cup\{(x,y\in H\,:\,x>1\text{ and }y<2x-2\}$ and $U_\varnothing=\{(x,y)\in H\,:\,(x,2y)\in T\}$. 
We define
\begin{equation*}
\rho(z)=
\begin{cases}
&d(z)^{-1},\text{ if }z\in\cup_{\iii} U_\iii\,\\
&d(z),\text{ if }z\notin\cup_{\iii} T_\iii\,, 
\end{cases}
\end{equation*}
where the union is over all $\iii\in\{\varnothing\}\cup_{n\in\N}\{0,1\}^n$
Moreover, we extend $\rho$ continuously into the strips $T_{\iii}\setminus U_{\iii}$ such that it is monotone in the $y$-coordinate. 

It is now easy to see that $\partial_\rho H=C$ and that $i^+=1$ on $\partial_\rho H$. Since $\sum_{n}\alpha_n<\infty$, it follows that $\Le(C)>0$ and thus in particular $\dim_d(C)=\Dim_d(C)=1$. If $n\in\N$ and $\iii\in\{0,1\}^n$, we can connect any two points of $C\cap I_\iii$ by two vertical segments of length $h_n$ and the horisontal segment between their tops such that apart from endpoints, these segments lie completely outside $\cup_{\iii}T_\iii$. This implies $\diam_\rho(C\cap I_\iii)\le  h_{n}^2+\ell_n h_n\le 2 \ell_n h_n$ and thus for each $n$, there is a covering of $\partial_\rho H$ by $2^n$ sets of $\rho$-diameter $2\ell_n h_n$. 
Combining with \eqref{dim0estimate} and letting $n\rightarrow\infty$ yields $\dim_\rho(\partial_\rho H)=\Dim_\rho(\partial_\rho H)=0$. This shows that the claims \eqref{1} and \eqref{3} of Lemma \ref{thm:estimates1} are not valid. 
\end{ex}

The final example of this section shows that neither the estimates \eqref{2}--\eqref{4} of
Lemma \ref{thm:estimates1} nor \eqref{est23}--\eqref{est24} of Theorem
\ref{estimates2} need hold if
$\Omega$ is not a John domain.

\begin{ex}\label{ex:Koch}
We construct a snowflake type domain
$\Omega\subset\R^2$ that does not satisfy \eqref{2} nor
\eqref{4} of Lemma \ref{thm:estimates1}. 

To begin with, we fix $0<s<1/2$ and let
$0<\alpha_1<1/2$.
We
start with an equilateral triangle with sides of length $l_0=1$ and
replace the middle $\alpha_1$-th portion of each of the sides by two
segments of length $l_1 = (1-\alpha_1)/2$. 
We continue inductively.
At the step $k$, we have $3\cdot 4^{k}$
segments of length $l_k$ and we replace the middle $\alpha_k$-th
portion of each of these segments by two line segments of length
$l_{k+1}=l_k(1-\alpha_{k+1})/2$, see Figure \ref{fig:koch1}. The numbers
$\alpha_k$ are defined so that
\begin{equation}\label{akdef}
\alpha_{k+1}=l_{k}^{1-2s}(1-\alpha_{k+1})/2\,.
\end{equation}
Observe that $\alpha_k\downarrow 0$ as $k\rightarrow\infty$.
We denote by $\Omega_k$ the domain bounded by the
line segments at step $k$ and define
$\Omega=\cup_{k\in\N}\Omega_k$.
We denote by $\Sigma\subset\intbndryOmega$ the part of the boundary
that joins two vertexes of the original equilateral triangle and does
not contain the third vertex. 
For notational convenience, we consider only points of
$\Sigma$. This does not affect the generality as
$\intbndryOmega\setminus\Sigma$ consist of two translates of
$\Sigma$. For $x\in\Sigma$, we let
$a(x)\in\{1,2,3,4\}^\N$ denote its coding or ``address'' arising from the
enumeration of the segments in each level as in the
Figure \ref{fig:koch1}. Note that this address is unique outside a
countable set of
points.

\begin{figure}
    \begin{center}
      \includegraphics[width=1.0\textwidth]{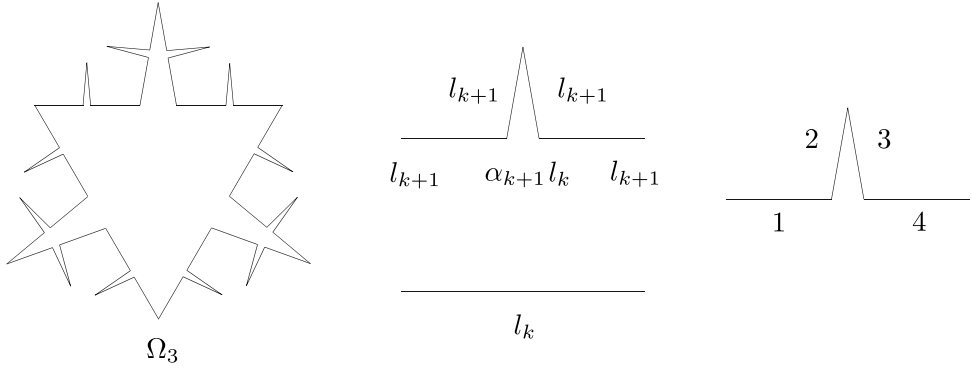}
    \end{center}
    \caption{Three figures concerning Example \ref{ex:Koch}: domain
      $\Omega_3$ (left), construction at level $k+1$ (middle) and
      enumeration of the segments (right).}
    \label{fig:koch1}
\end{figure}

Next we define
$\rho(z)=d(z)^{-1/2}$ for all $z\in\Omega$ and consider the set
$A=\{x\in\Sigma\,:\,a(x)\in\{2,3\}^\N\}$. It is easy to see
that there are numbers $0<D_1<D_2<\infty$, so that
$\dim_d(A)=D_1=\Dim_d(A)$ and $\dim_d(\intbndryOmega)=D_2=\Dim_d(\intbndryOmega)$ (actually
$D_1=1$ and $D_2=2$ but this is not essential). If we show that
\begin{equation}\label{Adim}
\dim_\rho(A)=\Dim_\rho(A)=D_1/s,
\end{equation}
then it follows that the claims \eqref{2} and
\eqref{4} of Lemma \ref{thm:estimates1} do not hold. Observe that
$i^-(x)=i^+(x)=-1/2$ for all $x\in\intbndryOmega$.

Let $x\in A$ and
$y\in\closuredOmega$
and choose the smallest $k\in\N$ so that
$l_k<2d(x,y)$. Let $z$ be as in Figure
\ref{fig:koch2}, \emph{i.e.} $z$ is the ``base point'' of a cone of $\Omega_k$
with ``side-length'' $l_k$ which is closest to $x$. Then
\begin{figure}
    \begin{center}
      \includegraphics[width=0.7\textwidth]{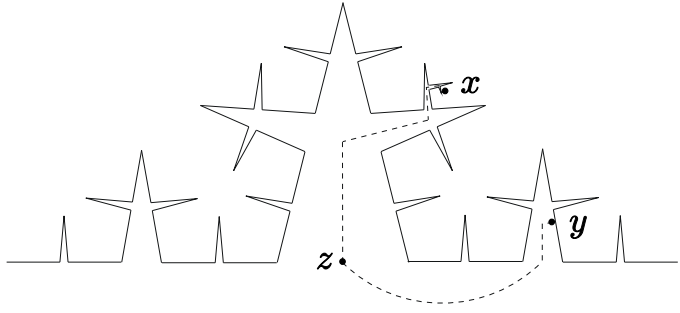}
    \end{center}
    \caption{Selection of the``base point" $z$ in Example \ref{ex:Koch}.}
    \label{fig:koch2}
\end{figure}
\begin{align*}
d_\rho(x,z)\leq
c_0\sum_{n=k}^\infty \alpha_n^{-1/2} l_{n}^{1/2}=c_0
\sum_{n=k}^{\infty}l_{n-1}^s\leq c_1 l_{k}^s\leq 2^s c_1 d(x,y)^s.
\end{align*}
Here the first equality follows from \eqref{akdef} and the former estimate holds because $l_k/4< l_{k+1}<l_k/2$ for all
$k\in\N$. By a similar argument, it follows that $d_{\rho}(z,y)\leq c_2
l_{k}^s\leq c_3
d(x,y)^s$. Thus $d_\rho(x,y)\leq c d(x,y)^s$.
On the other hand, it is clear that
$d_\rho(x,y)\geq c_4 l_{k}^s\geq c' d(x,y)^s$, since $a(x)\in\{2,3\}^\N$.
Thus, we have $c' d_\rho(x,y)\leq d(x,y)^s\leq
 c d_\rho(x,y)$, in other words $B_\rho(x,c'r)\subset B_d(x,r^s)\subset
B_\rho(x,cr)$, for all $x\in A$ and $y\in\closuredOmega$ where the
constants $0<c',c<\infty$ are independent of the points $x$ and $y$.
The claim \eqref{Adim} now follows from Lemma \ref{dimlemmarho}.
\end{ex}

\begin{rem}
Suppose that  $A\subset
\{x\in\partial_\rho\Omega\,:\,i^{-}(x)\geq\beta\}$ has the
following accessibility property for some $1\leq \lambda<-1/\beta$: For
each $x\in\intbndryOmega$ there are $0<r,c<\infty$ such that for all $y\in
B_d(x,r)\cap A$ there exists a curve $\gamma$ joining $x$ and
$y$ so that $d(\gamma(t),\intbndryOmega)\geq c\min\{t^\lambda,
(\length(\gamma)-t)^\lambda\}$  for all $0<t<\length(\gamma)$. Then the
proof of Lemma \ref{thm:estimates1} with trivial modifications implies
$(1+\lambda\beta)\dim_\rho(A)\leq\dim_d(A)$ and
$(1+\lambda\beta)\Dim_\rho(A)\leq\Dim_d(A)$. On the other hand, if
for each $x\in B\subset
\{x\in\partial_\rho\Omega\,:\,i^+(x)\leq\beta\}$ there are
$0<r,c<\infty$ so that for all curves $\gamma$ with $\gamma(0)=x$ we
have $d(\gamma(t),\intbndryOmega)<c t^\lambda$ for $0<t<r$, then we
get $(1+\lambda\beta)\dim_\rho(B)\geq\dim_d(B)$,
$(1+\lambda\beta)\Dim_\rho(B)\geq\Dim_d(B)$. The previous example
shows that these estimates are sharp.
\end{rem}

\section{Conformal densities}\label{conf}

The results in the last section, are based on estimates of the
quantities $i^+(x)$ and $i^-(x)$ which are defined as internal
limits when $\Omega\ni y\rightarrow x\in\intbndryOmega$. This causes a
lack of the generality; it is quite possible that $i^{+}(x)=0$ and
$i^{-}(x)=-1$ for all $x\in\intbndryOmega$. (For instance, choose
$\beta=-1$, $\lambda=0$ in the forthcoming Example \ref{ex:multifractal}.) 
However, if we have additional information on the geometry of
$( \closurerhoOmega ,d_\rho)$, then it might be enough to consider the
ratios $\log\rho(y)/\log d(y)$ along some fixed curves or cones.
The purpose of this section is to show that this is the case
for so called conformal densities which
arise naturally in connection with conformal and quasiconformal
mappings and their generalisations, 
see \cite{BKR}. 

A density $\rho$ on $\mathbb{B}^n$ is called a conformal density if
there are constants $1\le c_0, c_1\le\infty$ such that for each
$x\in\mathbb{B}^n$ and for all $y \in B_e(x,d(x)/2)$ we have
\begin{equation}\label{HI}
c_{0}^{-1}\leq\rho(y)/\rho(x)\leq c_0,
\end{equation}
and moreover,
\begin{equation}\label{VG}
\mu_\rho(B_\rho(x,r))\leq c_1 r^n
\end{equation}
for all $r>0$.
Here $\mu_\rho$ is the measure given by
$\mu_\rho(E)=\int_{E}\rho^n\,d\Le^n$ for $E\subset \mathbb{B}^n$. 
In the literature, \eqref{HI} is often called the Harnack inequality, and one refers 
to \eqref{VG} as a volume growth condition.     
An important corollary of the conditions \eqref{HI}--\eqref{VG} is the following Gehring-Hayman type estimate: There exists $1\le
c<\infty$ such that 
\begin{equation}\label{GHtype2}
c^{-1}d_\rho(x,y)
\leq\int_{t=0}^{d(x,y)}\rho\left((1-t)x\right)\,dt
+\int_{t=0}^{d(x,y)}\rho\left((1-t)y\right)\,dt
\leq c d_\rho(x,y)
\end{equation}
for all $x,y\in\partial_\rho\mathbb{B}^n$.
See \cite[Theorem 3.1]{BKR} and also \cite{GH}.

Motivated by this estimate, we consider variants $k^-$ and $k^+$
of the quantities $i^-$ and $i^+$ for a density $\rho$ on
$\mathbb{B}^n$ at $x\in S^{n-1}$. Recall that
$k^-(x)=\liminf_{r\uparrow1}\log\rho(rx)/\log(1-r)$, and
$k^+(x)=\limsup_{r\uparrow1}\log\rho(rx)/\log(1-r)$.
Occasionally we also use $k^-$ and $k^+$ when $\Omega=\mathbb{H}$ is an open half
space and then the limits are considered along straight lines
orthogonal to the boundary of $\mathbb{H}$. The reduction to $k^\pm$ is possible since \eqref{GHtype2} is a much stronger condition than the Assumption \ref{assu:gamma} that was used earlier for the same purpose.

In the following result we only assume that
\eqref{HI} and \eqref{GHtype2} hold. Thus, the result applies to a
slightly larger collection of densities than the conformal
densities. See \cite{NT}, and also Example \ref{ex:multifractal}
to follow. 

\begin{thm}\label{thm:estimates4} Suppose that $\rho$ is a density on
  $\mathbb{B}^n$ that satisfies the conditions \eqref{HI} and
  \eqref{GHtype2}. Let $\beta>-1$,
\begin{align*}  
A\subset\{x\in\partial_\rho\mathbb{B}^n\,:\,k^{+}(x)\le\beta\},\\
B\subset\{x\in\partial_\rho\mathbb{B}^n\,:\,k^{-}(x)\ge\beta\},\\
C\subset\{x\in\partial_\rho\mathbb{B}^n\,:\,k^{-}(x)\le\beta\}. 
\end{align*}
Then
\begin{enumerate}
\item\label{C1}
  $(1+\beta)\dim_\rho(A)\geq
  \dim_d(A)$,
\item\label{C2}
  $(1+\beta)\dim_\rho(B)\leq
  \dim_d(B)$,
\item\label{C3}
  $(1+\beta)\Dim_\rho(A)\geq
  \Dim_d(A)$,
\item\label{C4}
  $(1+\beta)\Dim_\rho(B)
  \leq\Dim_d(B)$,
\item\label{C5}
  $(1+\beta)\Dim_\rho(C)
  \geq\dim_d(C)$.
\end{enumerate}
\end{thm}

\begin{proof}
The claims \eqref{C1}--\eqref{C4} have proofs very similar to the
proofs of the corresponding statements of Lemma
\ref{thm:estimates1}. 
We first apply \eqref{GHtype2} to conclude that for each
$x\in A$ and $y\in \mathbb{B}^n\setminus
B_d(x,r)$, we have
\[d_\rho(x,y)\geq c^{-1}\int_{t=0}^r\rho((1-t)x)\,dt\geq
c^{-1}\int_{t=0}^r t^s\ge c_0 r^{1+s}\]
if $s>\beta$ and $r>0$ is small. This implies
$\diam_d(B_\rho(x,r))\leq c_1 r^{1/(1+s)}$ and the claims \eqref{C1} and
\eqref{C3} now follow by Lemma \ref{dimlemmarho} \eqref{rholemma2}.

To prove \eqref{C2} and \eqref{C4}, let $s<\beta$
and for $n\in\N$, denote 
\[B_n=\{x\in B\,:\,\rho((1-t)x)<t^s\text{ for all }0<t<1/n\}.\]
Using
\eqref{GHtype2}, we find $r_0>0$ so that
\[d_\rho(x,y)\leq c_2\int_{t=0}^{d(x,y)}t^s\le c_3 d(x,y)^{1+s}\] 
whenever $x,y\in B_n$ and $d(x,y)<r_0$. In other words,
$\diam_\rho(B_d(x,r)\cap B_n)\leq c_4 r^{1+s}$ when
$0<r<r_{0}^{1/(1+s)}$. Now the Lemma \ref{dimlemmarho} \eqref{rholemma1}
implies $(1+s)\dim_\rho(B_n)\leq\dim_d(B_n)$ and
$(1+s)\Dim_\rho(B_n)\leq\Dim_d(B_n)$. Note that it is enough to assume
$\liminf_{r\downarrow
    0} (\log\left(\diam_\rho(B_d(x,r)\cap A)\right))/(\log r) \geq \delta$ in
  Lemma \ref{dimlemmarho} \eqref{rholemma1} (since
  $\dim_{\partial_\rho\mathbb{B}^n}(A)=\dim_{(A,d_\rho)}(A)$ and
  $\Dim_{\partial_\rho\mathbb{B}^n}(A)=\Dim_{(A,d_\rho)}(A)$). The claims
  \eqref{C2} and \eqref{C4} now follow since $B=\cup_{n\in\N}$ and
  $s<\beta$ is arbitrary.

It remains to prove \eqref{C5}. Let $x\in C$ and $s>\beta$. Then
there is a sequence $0<r_i\downarrow0$ such that
$\rho((1-r_i)x)>r_{i}^s$ for all $i$. Combined with \eqref{HI}, this
gives 
\[\int_{t=0}^{r_i}\varrho((1-t)x)\,dt\ge c_5 r_{i}^{1+s}\]
and using also \eqref{GHtype2}, $\diam_d(B_\rho(x,c_6 r_{i}^{1+s}))\le
r_i$. Thus 
\[\limsup_{r\downarrow0}\frac{\log\diam_d(B_\rho(x,r))}{\log
r} \ge \frac{1}{1+s}\] 
and \eqref{C5} follows from Lemma \ref{dimlemmarho} \eqref{rholemma4}. 
\end{proof}

\begin{rems}\label{crems}
\textit{a)} Using the claims \eqref{C1}--\eqref{C4} of Theorem
\ref{thm:estimates4} one may derive multifractal type
formulas completely analogous to \eqref{est21}--\eqref{est24}.
Using
\eqref{C5}, we have moreover, that $\Dim_\rho(\partial_\rho
B_d(0,1))\geq\sup_{\beta>-1}\frac{e^-(\beta)}{1+\beta}$ where
\begin{equation}\label{limite}
  e^-(\beta)=\dim_d (\{x\in\partial_\rho
B_d(0,1)\,:\,k^-(x)\le \beta\}).
\end{equation}
Example \ref{hassu} shows that this is sharp
in the sense that one can not replace $\dim_d$ by $\Dim_d$ in defining
$e^-(\beta)$ even if $\rho$ is a conformal density.\\
\textit{b)} We formulated the above result for densities defined on
$\mathbb{B}^n$. The same proof goes through for any John domain
$\Omega\subset\Rn$ if the condition \eqref{GHtype2} is replaced by
\[c^{-1}d_\rho(x,y)\leq\int_{t=0}^{d(x,y)}\rho(\gamma_x(t))\,dt+\int_{t=0}^{d(x,y)}\rho(\gamma_y(t))\,dt\le
c d_\rho(x,y),\]
where $\gamma_x$ is a fixed $\alpha$-cone with $\gamma_x(0)=x$ for
each $x\in\partial_\rho\Omega$. Actually, we could even weaken this
in the spirit of \eqref{assu:gamma} and assume only that for all
$\varepsilon>0$, we have
\[d_\rho(x,y)^{1+\varepsilon}\leq\int_{t=0}^{d(x,y)}\rho(\gamma_x(t))\,dt+\int_{t=0}^{d(x,y)}\rho(\gamma_y(t))\,dt\le
d_\rho(x,y)^{1-\varepsilon}\]
when $d(x,y)$ is small enough.\\
\textit{c)} Makarov \cite[Theorems 0.5, 0.6]{Mak89} proved results essentially similar to Theorem \ref{thm:estimates4} \eqref{C1}--\eqref{C2} in case $\beta>0$ and $\rho=|f'|$ for $f$ conformal. He also showed \cite[Theorem 0.8]{Mak89} that $k^-$ cannot be replaced by $k^+$ in \eqref{C2}.\\ 
\textit{d)} In \cite{BKR}, Bonk, Koskela, and Rohde proved the following deep
fact. If $\rho$ is a
conformal density on $\mathbb{B}^n$, then: 
\begin{equation}\label{BKR}
\text{There is }E\subset S^{n-1}\text{ with }\dim_d E=0\text{ such that
}\dim_\rho(\partial_\rho \mathbb{B}^n\setminus E)\leq n. 
\end{equation}
See \cite[Theorem 7.2]{BKR}. As a central
tool, they used an estimate analogous to Theorem
\ref{thm:estimates4} \eqref{C2}. In fact, combining Theorem \ref{thm:estimates4} \eqref{C2} and \cite[Theorem
5.2]{BKR} gives a simpler proof for \eqref{BKR} than the one given in
\cite{BKR}. However, their result is
quantitatively stronger than \eqref{BKR}.\\
\textit{e)} A generic situation in which Theorem \ref{thm:estimates4}
is stronger than Theorem \ref{estimates2} will be discussed in Example \ref{ex:multifractal}.
\end{rems}

\section{Further examples, remarks,  and questions}
\label{erq}

We first give the example mentioned in Remark \ref{crems} e) showing
that one can not replace $\dim_d$ by $\Dim_d$ in defining $e^-(\beta)$. We will make use of the following lemma. We formulate it in a more general setting, for future reference.

\begin{lemma}\label{lemma:aux}
Let $\Omega\subset\R^n$ be a $(2\alpha)$-John domain and $C\subset\intbndryOmega$. Suppose that $\widetilde{\rho}\colon(0,\infty)\rightarrow(0,\infty)$ is nonincreasing and and satisfies $\int_0^1 \widetilde{\rho}(t)dt < \infty$. Define $\rho(x)=\widetilde{\rho}(d(x,C))$ for $x\in\Omega$. Then for all $x\in C$ and
$0<r<\diam_d(\Omega)/2$, it holds
\begin{align}
\diam_d \left(B_\rho\left(x,\frac12\int_{t=0}^r\widetilde{\rho}(t) dt\right)\right)\leq 2r,\label{akuankka}\\
\diam_\rho\left(B_d(x,r)\right)\leq c_1 \int_{t=0}^{c_2 r}\widetilde{\rho}(t)\,dt\label{hessuhopo}
\end{align}
for some constants $0<c_1,c_2<\infty$ that depend only on $\alpha$ and
$n$.
\end{lemma}

\begin{proof}
Let $x\in C$ and $y\in\closuredOmega$. Denote $d=d(x,y)$ and suppose that $\gamma$ is
a curve joining $x$ and $y$. To prove \eqref{akuankka}, it suffices to
show that
\begin{equation}\label{tupuankka}
\length_\rho(\gamma)\ge\frac12 \int_{t=0}^{d}\widetilde{\rho}(t)\,dt.
\end{equation}
Let $h=h(\gamma)=\max_{0\leq t\leq
  L}d(\gamma(t))$ where $L=\length(\gamma)$. Then
$\length_\rho(\gamma)\geq\tfrac12 \int_{t=0}^h\widetilde{\rho}(t)\,dt+\tfrac12 d\rho(h)$. If
$h\geq d$ the estimate \eqref{tupuankka} clearly holds. If $h<d$,
then $d\widetilde{\rho}(h)\geq\int_{t=h}^{d}\widetilde{\rho}(t) dt$ since $\widetilde{\rho}$ is nonincreasing
and consequently
\[\length_\rho(\gamma)
\geq\frac12\left(\int_{r=0}^{h}\widetilde{\rho}(r)\,dr+d\widetilde{\rho}(h)\right)
  \ge\frac12\int_{t=0}^d\widetilde{\rho}(t)\,dt.\]
This settles the proof of \eqref{akuankka}.

To prove \eqref{hessuhopo}, let $x\in C$ and $r>0$. We use
Lemma \ref{lemma:john} to cover  
$B_d(x,r)$ with sets
$B_{\alpha}(x_i,c r)$, $i=1,\ldots,N=N(n,\alpha)$. Let $y\in
B_{\alpha}(x_i,c r)$
and pick an $\alpha$-cigar $\gamma$ with $\length(\gamma)\leq c r$
joining $y$ to $x_i$. Now
\begin{equation}\label{hessuhopotod}
d_\rho(y,x_i) \leq \int_\gamma\widetilde{\rho}(d(z,C))|dz|\leq 2\int_{t=0}^{c
  r/2}\widetilde{\rho}(\alpha t)\,dt=\frac2\alpha\int_{t=0}^{\alpha c
  r/2}\widetilde{\rho}(t)\,dt.
\end{equation}
As $B_d(x,r)$ is (path-)connected and is covered by $N$ sets of the type
$B_\alpha(x_i,c r_i)$, we arrive at
$\diam_\rho(B_d(x,r))\leq(4N/\alpha)\int_{t=0}^{\alpha
  cr/2}\widetilde{\rho}(t)\,dt$
proving the claim.
\end{proof}

\begin{ex}\label{hassu}
We show that $\dim_d$ cannot be replaced by $\Dim_d$ in \eqref{limite}
even if $\rho$ is a conformal density. 

We first fix numbers $0<a<b<1/2$,
$-1<\lambda<\eta<0$, and $\xi$ such that
\begin{equation}\label{eq:viimonen}
a^{1+\lambda}=b^{1+\eta}=\xi, 
\end{equation}
and
\begin{equation}\label{kaikkiseuraatasta}
-\log 2<\log\xi<-\tfrac12\log 2.
\end{equation}
Let us also pick natural numbers $n_1<N_1<n_2<N_2<n_3<N_3<\ldots$. We let
$C\subset S^1$ denote a Cantor set constructed as follows (See the construction in Example \ref{ex:assugamma}). We start
with an arc of length $1$ and remove an arc of length $1-2a$ from
the middle. Next, we remove arcs of length $a(1-2a)$ from
the middle of the two remaining arcs. We iterate this
construction for $n_1$ steps. After these $n_1$ steps, we have
$2^{n_1}$ arcs of length $a^{n_1}$. At the step $n_1+1$,
we remove arcs of relative length $1-2b$ from the middle of
each of these arcs. We continue the construction with the
parameter $b$ for $N_1-n_1$ steps. Then we use again the
parameter $a$ for $n_2-N_1$ steps and so on. 
We denote by $E_{k,1}\ldots, E_{k,2^k}$ the arcs remaining after
$k$ steps and denote by $\ell_k$ the length of these arcs.
What remains at the
end is the Cantor set $C=\cap_{k\in\N}\cup_{i=1}^{2^k}E_{k,i}$. 

Let $r_0=R_1=1$, $r_1=\ell_{n_1}=a^{n_1}$,
$R_2=\ell_{N_1}=a^{n_1}b^{N_1-n_1}$,
$r_2=\ell_{n_2}=a^{n_1+n_2-N_1}b^{N_1-n_1}$ and so
on. Thus $r_i$ (resp. $R_i$) is the length of a construction interval
of $C$ of level $n_i$ (resp. $N_{i-1}$). We define
$\rho(x)=\widetilde{\rho}(\dist(x,C))$ for all $x\in\mathbb{B}^2$,
where $\widetilde{\rho}$ is the function defined by
\begin{equation*}
\widetilde{\rho}(t)=\begin{cases}
\left(\frac{R_1 R_2\cdots R_k}{r_0 r_1\cdots
  r_{k-1}}\right)^{\eta-\lambda}t^{\lambda},\quad r_k\leq t\leq R_k,\\ 
\left(\frac{R_1 R_2\cdots R_k}{r_0 r_1\cdots
    r_{k}}\right)^{\eta-\lambda}t^{\eta},
\quad R_{k+1}\leq t\leq r_k.
\end{cases}
\end{equation*}

Now, if $N_i/n_i, n_{i+1}/N_i\rightarrow\infty$ fast enough, it
is easy to see that $\dim_d C=-\log 2/\log a$ and $\Dim_d C=-\log
2/\log b$, see e.g. \cite[p. 77]{M}. Moreover, it then
follows that $k^-(x)=\lambda$ if $x\in C$ and $k^-(x)=0$
otherwise. 
Next, let $h_k=\int_{t=0}^{\ell_k}\widetilde{\rho}(t)\,dt$. Since
\begin{equation}\label{tilderho}
\widetilde{\rho}(\ell_k)\ell_k=\xi^k
\end{equation}
for all $k$ (combine \eqref{eq:viimonen} with the definitions), it follows that
\[\tfrac12 \xi^k=\tfrac12\widetilde{\rho}(\ell_k)\ell_k\le\int_{t=\ell_{k+1}}^{\ell_k}\widetilde{\rho}(t)
\le\widetilde{\rho}(\ell_{k+1})\ell_k\le a^\lambda\widetilde{\rho}(\ell_k)\ell_k
=a^\lambda \xi^k.\] 
Thus
\begin{equation}\label{ynnatasta}
\tfrac12 \xi^k\le h_k =\sum_{m\ge
  k}\int_{t=\ell_{m+1}}^{\ell_m}\widetilde{\rho}(t)\, dt\le c_0 \xi^k.
\end{equation}
From Lemma \ref{lemma:aux}, it follows that for each
$I=I_{k,i}$ we have
\begin{equation}\label{jatasta}
c_1 h_k\le\diam_{\rho}(I)\le c_2 h_k
\end{equation}
for some constants $0<c_1<c_2<\infty$. Let $\mu$ be the natural
probability measure on $C$ that satisfies $\mu(I_{k,i})=2^{-k}$. Then
\[\lim_{k\rightarrow\infty}\frac{\log\mu(I_{k,i})}{\log(\diam_\rho(I_{k,i}))}= \frac{-\log 2}{(1+\lambda)\log a},\]
using \eqref{ynnatasta} and \eqref{jatasta}. But this implies
$\dim_\rho(C)=\Dim_\rho(C) = (-\log 2)/((1+\lambda)\log a)$, see e.g. \cite[Proposition
10.1]{Falconer1997} and \cite[Corollary 3.20]{Cutler1995}. Thus, 
\begin{align*}
1<\Dim_\rho(C)&=\Dim_\rho(S^1)=\frac{-\log 2}{(1+\lambda)\log
  a}<\frac{-\log 2}{(1+\lambda)\log b}=\frac{\Dim_d(C)}{1+\lambda}\\
&=\sup_{\beta>-1}\frac{\Dim_d(\{x\in\partial_\rho
B_d(0,1)\,:\,k^-(x)\le \beta\}}{1+\beta}\,,
\end{align*}
recall \eqref{kaikkiseuraatasta}.

It remains to prove that $\rho$ is a conformal density. The condition
\eqref{HI} is clearly satisfied so we only have to verify
\eqref{VG}. We show this for $x\in C$ and $0<r<1$ (the general case
$x\in\mathbb{B}^2$ follows easily from this). Using \eqref{HI} we
may also assume that
$r=h_k$ for some $k\in\N$. For each $m\ge k$, we denote
\[A_m=\{y\in B_d(x,c_3\ell_k)\,:\,\ell_m\le d(y,C)\le c_3\ell_m\}.\]
Then $B_\rho(x,h_k)\subset\cup_{m\ge k}A_m$, for a suitable constant
$1<c_3<\infty$, recall \eqref{jatasta}. Moreover, it follows from
\eqref{HI} and \eqref{tilderho} that 
$c_4\xi^m/\ell_m\le\rho(y)\le c_5\xi^m/\ell_m$
for all $y\in A_m$,
where $0<c_4<c_5<\infty$ depend only on $a,b, \lambda$, and
$\eta$. Since $\mathcal{L}^2(A_m)\le c_6 2^{m-k}\ell_{m}^2$,
we arrive at
\[\mu_\rho(A_m)=\int_{A_m}\rho^2\,d\mathcal{L}^2\le c_7 2^{m-k}\xi^{2m}.\]
As $2\xi^{2}<1$ by \eqref{kaikkiseuraatasta}, this yields
\begin{align*}
\mu_\rho(B_\rho(x,h_k))\le\sum_{m\ge k}\mu_\rho(A_m)\le c_7 \sum_{m\ge k}2^{m-k}\xi^{2m}
\le c_8 \xi^{2k}\le c_9 h_{k}^2,
\end{align*} 
where the last estimate follows from \eqref{ynnatasta}.
\end{ex}

Below, we construct a ``multifractal type'' example and calculate the
Hausdorff dimension of the boundary using Theorem \ref{thm:estimates4}.

\begin{ex}\label{ex:multifractal}
We construct a domain and a conformal density that satisfies Gehring-Hayman condition \eqref{GHtype2} and compute the Hausdorff dimension of the boundary.

We define a density $\rho$ on the upper half-plane $H\subset\R^2$
(actually we
define $\rho(z)$ only for $z\in[0,1]\times(0,3]$ but the
definition is easily extended to the whole of $H$). Let
$-1<\beta,\lambda<0$, $\beta\neq\lambda$. 
We consider the triadic decomposition of $[0,1]$; Let $I_{\varnothing}=[0,1]$,
$I_{0}=[0,1/3]$, $I_{1}=[1/3,2/3]$, and $I_{2}=[2/3,1]$.
If $n\in\N$ and, $\textbf{i}\in\{0,1,2\}^n$, let $I_{\textbf{i}0},
I_{\textbf{i}1}, I_{\textbf{i}2}$ denote its triadic subintervals
  enumerated from left to right. For each such triadic interval
  $I=I_\textbf{i}$, let $Q_\textbf{i}=I\times\left[|I|,3|I|\right]$.
Next we define weights $\rho_\textbf{i}$ inductively by the rules
$\rho_\varnothing=1$ and
$\rho_{\textbf{i}0}=\rho_{\textbf{i}2}=3^{-\lambda}\rho_{\textbf{i}}$,
$\rho_{\textbf{i}1}=3^{-\beta}\rho_\textbf{i}$.

Let $\rho\colon[0,1]\times(0,3]\rightarrow(0,\infty)$ be a density such that 
$\rho(x_\textbf{i})=\rho_\textbf{i}$ if $x_\textbf{i}$ is the centre
point of $Q_\textbf{i}$. We also require that the condition
\eqref{HI} holds with some $c_0<\infty$. This is possible because
of the symmetric definition of $\rho_\textbf{i}$: 
If $I_\textbf{i}$ and $I_\textbf{j}$ are neighbouring intervals of the
same length, then
$3^{-|\beta-\lambda|}\leq|\rho_\textbf{i}/\rho_\textbf{j}|\leq
3^{|\beta-\lambda|}$.

We will next show that the Gehring-Hayman condition \eqref{GHtype2} holds for the
density $\rho$. Let $x,y\in[0,1]$ with $y-x=r>0$.
Let $\gamma_1$, $\gamma_2$, and $\gamma_3$ be the line segments
joining $(x,0)$ to $(x,r)$, $(x,r)$ to $(y,r)$, and 
$(y,r)$ to $(y,0)$, respectively. Then a direct calculation using the
definitions gives
\begin{eqnarray*}
\int_{\gamma_1}\rho(z)\,|dz|\le
c_1\int_{t=0}^rt^{\min\{\beta,\lambda\}}\frac{\rho(x,r)}{r^{\min\{\beta,\lambda\}}}\,dt\le 
c_2 r \rho(x,r),\\
\int_{\gamma_3}\rho(z)\,|dz|\le 
c_1\int_{t=0}^rt^{\min\{\beta,\lambda\}}\frac{\rho(y,r)}{r^{\min\{\beta,\lambda\}}}\,dt\le
c_2 r \rho(y,r).
\end{eqnarray*}
Combining these estimates with \eqref{HI}, we obtain
\begin{equation}\label{gammaest1}
c_3 \length_\rho(\gamma_i)\le \length_\rho(\gamma_2)\le c_4 \length_\rho(\gamma_i)
\end{equation}
for $i=1,3$. The condition \eqref{GHtype2} is satisfied if we can show
that $\length_\rho(\gamma)\ge
c\length_\rho(\gamma_2)$ for any curve joining $x$ and $y$ in $H$. Denote
$h=h(\gamma)=\max_{0<t<\length(\gamma)}d(\gamma(t))$. If $h\le r$, it follows
that $\length_\rho(\gamma)\ge c\length_\rho(\gamma_2)$ since $\rho$ is
essentially decreasing on each vertical line segment. More precisely
using \eqref{GHtype2} and the definitions of the weights
$\rho_\textbf{i}$, we get
\begin{equation}\label{rhoest}
\rho(a,tb)\ge c_5\rho(a,b)
\end{equation}
if $(a,b)\in[0,1]\times(0,3]$ and $0<t<1$. Now suppose
that $h>r$ and let $z=\gamma(t_0)$ where
$t_0=\min\{t>0\,:\,d(\gamma(t))=r\}$. If $d(z,\gamma_2)<r$, 
it follows easily from \eqref{HI} that $\length_\rho(\gamma)\ge
c\length_\rho(\gamma_2)$. If $d(z,\gamma_2)\ge r$, let $\eta$ be the
line segment joining $z$ to the closest point of $\gamma_2$. Then
\eqref{rhoest} implies $\length_\rho(\gamma)\ge c_5\length_\rho(\eta)\ge
c\length_\rho(\gamma_2)$ where the last estimate follows using
\eqref{HI}. This settles the proof of \eqref{GHtype2}. 

We will next compute the Hausdorff dimension of the boundary. Let $0\le t\le1$ and denote $A_t=\{x\in[0,1]\,:\,k^-(x)=k^+(x)=
t\beta+(1-t)\lambda\}$.
Then
\begin{align*}
A_t&= \left\{ x=\sum_{i\in\N}x_i 3^{-i}\,:\,x_i\in\{0,1,2\}\text{ and }
\lim_{n\rightarrow\infty}\#\{1\leq i\leq n\,:\,x_i=1\}/n=t\} \right\}.
\end{align*}
Using this expression, we get
\begin{equation}\label{selfsimdim}
\dim_d (A_t)=\Dim_d(A_t)=\frac{-t\log t+(t-1)\log((1-t)/2)}{\log 3}.
\end{equation}
Indeed, if $\mu_t$ is the unique Borel probability measure on $[0,1]$
that satisfies $\mu_t(I_{\textbf{i}1})=t\mu_t(I_\textbf{i})$ and
$\mu_t(I_{\textbf{i}0})=\mu_t(I_{\textbf{i}2})$ for all triadic intervals
$I_\textbf{i}$, then we have 
\[\lim_{r\downarrow 0}\frac{\log\mu_t((B_d(x,r))}{\log
r}=\frac{-t\log t+(t-1)\log((1-t)/2))}{\log 3}\] 
and this implies
\eqref{selfsimdim}. For instance, 
see \cite[Proposition 10.4]{Falconer1997}. 

Thus, from Theorem
\ref{thm:estimates4} and \eqref{selfsimdim}, 
we get  
\begin{equation}\label{maksimoitama}
\dim_\rho(A_t)=\Dim_\rho(A_t)=\frac{-t\log t
  +(t-1)\log((1-t)/2)}{(1+t\beta+(1-t)\lambda)\log 3}\,.
\end{equation} 
If  $f(\beta,\lambda)$ is the maximum of \eqref{maksimoitama} over
all $0\le t\le1$, then we conclude that 
\[\Dim_\rho(\partial_\rho H)\ge
\dim_\rho(\partial_\rho H)\ge f(\beta,\lambda).\] 

To finish this example, we show that for the Hausdorff dimension,
there is an equality in the above estimate. We give the proof in the
case $\beta<\lambda$, the case $\lambda<\beta$ can be handled with
similar arguments. First, we observe using Theorem
\ref{thm:estimates4} \eqref{C2} that 
\[\dim_\rho(\{k^-(x)\ge
\beta/3+2\lambda/3\})\le 1/(1+\beta/3+2\lambda/3)<f(\beta,\lambda),\]
where the strict inequality is obtained via 
differentiating \eqref{maksimoitama} at
$t=1/3$. On the other hand, if $t>1/3$, and $A^{-}_t=\{x\in[0,1]\,:\,k^-(x)\le
t\beta+(1-t)\lambda\}$, then 
\[A^{-}_t=\{x=\sum_{i\in\N}x_i 3^{-i}\,:\,
\limsup_{n\rightarrow\infty}\#\{1\leq i\leq n\,:\,x_i=1\}/n\ge t\}\}\]
and thus $\dim_d(A^{-}_t)\le (-t\log t+(t-1)\log((1-t)/2))/\log 3$. To
see this, observe that 
\[\liminf_{r\downarrow0}\frac{\log\mu_t(B_d(x,r))}{\log
r} \le \frac{-t\log t+(t-1)\log((1-t)/2))}{\log 3}\] 
for all $x\in A^{-}_t$ and
use \cite[Proposition 10.1]{Falconer1997}. Now, using the analogue of
\eqref{est23} for $k^{-}$ implies $\dim_\rho(\partial_\rho H)\le
f(\beta,\lambda)$, and consequently $\dim_\rho(\partial_\rho
H)=f(\beta,\lambda)$. 
\end{ex} 

\begin{rems}
\textit{a)} One can estimate the numbers $f(\beta,\lambda)$ numerically. For
instance, if $\beta=-1/2$ and $\lambda=-1/3$, then
$f(\beta,\lambda)\approx1.65$.\\ 
\textit{b)}
Inspecting \eqref{maksimoitama}, it is easy to 
see that 
\[\max\{1/(1+\beta/3+2\lambda/3),\log 2/((1+\lambda)\log
3)\}<f(\beta,\lambda)<1/(1+\min\{\beta,\lambda\})\] 
for all 
choices of $\beta$ and $\lambda$.\\   
\textit{c)} If above $\beta,\lambda>-1/2$, then it is not hard to see that
$\rho$ satisfies \eqref{VG} and
thus is a conformal density.
\end{rems}

We do not know if also
$\Dim_\rho(\partial_\rho H)\le f(\beta,\lambda)$:
\begin{q}
In Example \ref{ex:multifractal}, is it true that
$\Dim_\rho(\partial_\rho
H)=f(\beta,\lambda)$? 
\end{q}
We cannot use Theorem \ref{thm:estimates4} to solve this question
since it can be shown that $\Dim_d ( \{x\,:k^-(x)=\min\{\beta,\lambda\} ) =1$.

It is true that $\dim_\rho(\partial_\rho\mathbb{B}^n)\ge n-1$ for all
conformal densities $\rho$ defined on $\mathbb{B}^n$. This deep
fact was proved in \cite{BK}. A straightforward estimate using
Theorem \ref{thm:estimates4} and \eqref{HI} only implies that
$\dim_\rho(\partial_\rho\mathbb{B}^n)\ge c(n,c_0)>0$, where $c_0$ is the
constant in \eqref{HI}. See also \cite[Proposition 7.1]{BKR}. Next we
provide an example of a density $\rho$ on the upper half-plane $H$ such that
$\Dim_\rho(\partial_\rho H)=0$ and $\dim_d(\R\setminus\partial_\rho H)=0$.

\begin{ex}
We construct a density with $\Dim_\rho(\partial_\rho H)=0$ and $\dim_d(\R\setminus\partial_\rho H)=0$.

Given an interval $I\subset\R$, let $T_{I}$ and $U_I$ be the
isosceles triangles with base $I$ and heights $|I|$ and $|I|/2$
respectively. Denote
$S_{I}=T_{I}\setminus U_{I}$.

To begin with, let $I_1, I_2,\ldots$ be disjoint intervals so that
$C=\R\setminus\cup I_i$ forms a Cantor set (a nowhere dense closed set
without isolated points). Moreover, we assume that
$\sum_{i}\diam_d(I_i)\le 1$. Let $\rho(x)=\exp(-1/d(x))$ if
$x\in H\setminus\cup_{i}T_{I_i}$. We define
$\rho$ on each strip $S_{I_i}$ so that 
\begin{equation}\label{separation}
\ell_\rho(\gamma)\ge 1 
\end{equation}
for
any curve joining $U_{I_i}$ to $ H\setminus T_{I_i}$. 
We also require that $\rho$ extends continuously to the lower
boundary $\Gamma_{I_i}$ of $S_{I_i}$ (excluding the two endpoints of
$I_i$) and that 
\begin{equation}\label{glue1}
\ell_\rho(\gamma)=\infty
\end{equation}
if $\gamma$ is a curve on $S_{I_i}$ whose one endpoint is an
endpoint of $I_i$. We remark that the condition \eqref{glue1}
 as well as the condition \eqref{glue2} below, are only used to ensure
 that the assumption \eqref{ass2} is satisfied.

Now for
each $x,y\in C$ with $d(x,y)=d>0$, we have
\[d_\rho(x,y)\leq 2\int_{t=0}^d\exp(-1/t)\,dt+d\exp(-1/d)\le
3\exp(-1/d).\]
Thus, for each $n\in\N$, there is $\delta>0$ such that
$d_\rho(x,y)\le d(x,y)^n$ if $x,y\in C$ and $d(x,y)<\delta$. By Lemma
\ref{dimlemmaf}, this implies $\Dim_\rho(C)=0$.

We continue the construction inside the triangles $U_{I_i}$. We choose
intervals $I_{i,j}\subset I_i$ so that
$C_i=I_i\setminus\cup_{j}I_{i,j}$ is a Cantor set and
\begin{equation}\label{Cantordim}
\sum_{i,j}\diam_d(I_{i,j})^{1/2}\leq 1. 
\end{equation}
We define
$\rho(x)=f_{i}(x)\exp(-1/d(x))$ on $U_i\setminus\cup_{j}T_{I_{i,j}}$
where $f_{i}(x)$ is a continuous weight that is bounded if $x$ is
bounded away from the endpoints of $I_i$. Close to the
endpoints of $I_i$, we make $f_i$ so large that
\begin{equation}\label{glue2}
\ell_\rho(\gamma)=\infty
\end{equation}
if $\gamma$ is a curve on
$U_{I_i}$ whose one endpoint is an endpoint of $I_i$.
Also, we define $\rho$ on the strips $S_{I_i}$ so that
analogues of \eqref{separation} and \eqref{glue1} hold.
As above, we see that $\Dim_\rho(C_i)=0$ for all $i$.

We continue the construction inductively inside the triangles
$U_{I_{i,j}}$. At the step $n$, we obtain Cantor sets $C_{n,i}$
with $\Dim_\rho(C_{n,i})=0$. At the end, $\partial_\rho H$ will
be the union of all these Cantor sets. Replacing the exponent $1/2$ in
\eqref{Cantordim} by $1/n$ at the step $n$ implies that
$\dim_d(\R\setminus\partial_\rho H)=0$.
\end{ex}

It would be interesting to know, if the analogy of
\eqref{BKR} for the packing dimension holds.

\begin{q}
If $\rho$ is a conformal density on $\mathbb{B}^n$, does there exist a
set $A\subset S^{n-1}$ with $\Dim_d(A)=0$ such that
$\Dim_\rho(S^{n-1}\setminus A)\le n$? 
\end{q}

\textbf{Acknowledgements.} The first author was supported by the
Academy of Finland project \#120972 and he wishes to thank professor
Pekka Koskela. The second author was supported by the Academy of
Finland project \#126976.

\bibliographystyle{plain}
\bibliography{boundarydim}

\end{document}